\documentclass{article}

\usepackage{mystyle}
\usepackage{wasysym}
\usepackage{mathtools}
\begin{document}
\title{On products of ultrafilters}
\author{Gabriel Goldberg\\ Evans Hall\\ University Drive \\ Berkeley, CA 94720}
\maketitle
\begin{abstract}
    Assuming the Generalized Continuum Hypothesis, this paper answers
    the question: when is the tensor product of two ultrafilters
    equal to their Cartesian product?   
    It is necessary and sufficient that their Cartesian product is an ultrafilter;
    that the two ultrafilters commute in the tensor product;
    that for all cardinals \(\lambda\), one of the
    ultrafilters is both \(\lambda\)-indecomposable and \(\lambda^+\)-indecomposable;
    that the ultrapower embedding associated to each ultrafilter 
    restricts to a definable embedding of the ultrapower of the universe 
    assocaited to the other. 
\end{abstract}
\section{Introduction}
There are two binary operations on the class of ultrafilters
commonly referred to as the product, each of which fails 
to live up to the name.
The first, the Cartesian product, 
truly is a product in the category of filters.
The subcategory of ultrafilters, however,
is not closed under Cartesian products. 
The other, the tensor product, is at least
an operation on ultrafilters.
But it is not a product. It is not even commutative (up to isomorphism).
This raises the following questions:
\begin{qst}\label{qst:main_questions}\begin{enumerate}[(1)]
    \item When is the Cartesian product of two ultrafilters an ultrafilter?
    \item On which pairs of ultrafilters is the tensor product commutative?
    \item When do the Cartesian and tensor products coincide?
\end{enumerate}\end{qst}
For ultrafilters on \(\omega\),
the answer to all three questions is essentially \textit{never}. The
general question, however, is much more subtle and involves large cardinals.
The main conjecture motivating this paper is that all three questions have the same answer:
\begin{conj}
    The Cartesian product of two ultrafilters is an ultrafilter if and only if
    the ultrafilters commute under the tensor product if and only if 
    their Cartesian product is equal to their tensor product. 
\end{conj}

We now define the two products in order to state this conjecture more
precisely.
\begin{defn}
    If \(F\) and \(G\) are filters on sets \(X\) and \(Y\), the \textit{Cartesian product}
    of \(F\) and \(G\), denoted \(F\times G\),
    is the filter on \(X\times Y\) generated by sets of the form
    \(A\times B\) where \(A\in F\) and \(B\in G\). 
\end{defn}

The usual notation for the tensor product of ultrafilters
is \(\otimes\), but the following notation is more convenient here:
\begin{defn}
    Suppose \(F\) and \(G\) are filters on sets \(X\) and \(Y\).
    Their \textit{left and right tensor products} are defined as follows:
    \begin{align*}
        F\ltimes G &= \{R\subseteq X\times Y : \{x\in X : (R)_x\in G\}\in F\}\\
        F\rtimes G &= \{R\subseteq X\times Y : \{y\in Y : (R)^y\in F\}\in G\}
    \end{align*}
\end{defn}
Here \((R)_x = \{y\in Y : (x,y)\in R\}\) and \((R)^y = \{x\in X : (x,y)\in R\}\).

Our main conjecture now reads:
\begin{conj}\label{conj:main_conj}
    For any ultrafilters \(U\) and \(W\), the following are equivalent:
    \begin{enumerate}[(1)]
        \item \(U\times W\) is an ultrafilter.\label{item:conj:product}
        \item \(U\ltimes W = U\rtimes W\).\label{item:conj:commute}
        \item \(U\ltimes W = U\times W\). \label{item:conj:equal}
    \end{enumerate}
\end{conj}
The implication from \ref{item:conj:product} to \ref{item:conj:commute} 
and the equivalence of \ref{item:conj:product} and \ref{item:conj:equal} are easy to see:
\(U\times W\) is certainly contained in both \(U\ltimes W\) and \(U\rtimes W\),
so if \(U\times W\) is an ultrafilter, its maximality implies that
\(U\ltimes W = U\times W = U\rtimes W\).
There is really just one open question: 
\begin{qst}
    Suppose \(U\) and \(W\) are ultrafilters
    such that \(U\ltimes W = U\rtimes W\). Must \(U\times W\) be an ultrafilter?
\end{qst}
We will answer this question positively assuming the Generalized Continuum Hypothesis:
\begin{thm}[GCH]\label{thm:commutative_productive}
    Suppose \(U\) and \(W\) are ultrafilters. Then \(U\ltimes W = U\rtimes W\) if and only
    if \(U \times W\) is an ultrafilter.
\end{thm}
It is \textit{not}
true that if \(U\ltimes W\) and \(W\ltimes U\) are Rudin-Keisler equivalent, 
then \(U\times W\) is an ultrafilter.
For example, if \(U = W\), then \(U\times W\) is never an ultrafilter, but
obviously \(U\ltimes W = W\ltimes U\). So it is not enough that 
\(U\ltimes W\RKE W\ltimes U\); one also that this equivalence is witnessed by the
coordinate-swapping bijection from \(X\times Y\) to \(Y\times X\) where \(X\) and \(Y\) are the
underlying sets of \(U\) and \(W\).

The theorem is proved by isolating a combinatorial criterion that
arguably answers \cref{qst:main_questions} under GCH.
\begin{defn}\label{defn:indecomposability}
    A filter \(F\) is \textit{\((\lambda,\eta)\)-indecomposable} if for any family of sets
    sets \(\mathcal P\) of cardinality at most \(\eta\) with 
    \(\bigcup \mathcal P\in U\), there is some \(\mathcal Q\subseteq \mathcal P\)
    of cardinality less than \(\lambda\) such that \(\bigcup Q\in U\).
\end{defn}

\begin{repthm}{thm:indecomposable_product}
    Suppose \(U\) and \(W\) are ultrafilters such that for all cardinals \(\lambda\),
    either \(U\) or \(W\) is \((\lambda,2^\lambda)\)-indecomposable. 
    Then \(U\times W\) is an ultrafilter.
\end{repthm}

\begin{repthm}{thm:nonregular_indecomposable}
    Suppose \(U\) and \(W\) are ultrafilters such that \(U\ltimes W = U\rtimes W\). 
    Then for all cardinals \(\lambda\),
    either \(U\) or \(W\) is \((\lambda,\lambda^+)\)-indecomposable.
\end{repthm}

In addition, we will provide a third, somewhat more surprising condition
equivalent to the statement that \(U \times W\) is an ultrafilter. 
This involves
a variant of the Mitchell order called the \textit{internal relation}.
Recall that the Mitchell order is 
defined on countably complete ultrafilters \(U\) and \(W\)
by setting \(U\mo W\) if \(U\) belongs to the ultrapower of the universe by \(W\),
which will be denoted by \(M_W\).
The question
of whether \(U\in M_W\) only makes sense when \(W\) is countably complete,
so that \(M_W\) can be identified with a transitive class.
The internal relation is a 
combinatorial proxy
for the Mitchell order that in particular
makes sense regardless of 
the completeness of the ultrafilters in question. 
\begin{defn}\label{defn:internal_relation}
    The \textit{internal relation} is defined on
    ultrafilters \(U\) and \(W\) by setting \(U\I W\) if
    \(j_U\restriction M_W\) is isomorphic to an 
    internal ultrapower embedding of \(M_W\).\footnote{
        This means that there is an ultrafilter \(\tilde U\)
        of \(M_W\) and an isomorphism
        \(k : (M_{\tilde U})^{M_W}\to j_U(M_W)\)
        such that \(k\circ (j_{\tilde U})^{M_W}= j_U\).
    }
\end{defn}
Here \(j_U : V\to M_U\) denotes the ultrapower embedding,
as it will throughout the paper.

One advantage of the internal relation over
the Mitchell order is that the former relation is more amenable 
to algebraic or diagrammatic techniques.
This is exploited in the author's thesis to 
give a complete analysis of the internal relation 
on countably complete ultrafilters
assuming the \textit{Ultrapower Axiom}.
The Ultrapower Axiom is a combinatorial principle motivated by inner model theory
that enables one to develop a theory of countably complete
ultrafilters that is far more detailed than the
theory available assuming ZFC alone. 
The theory of countably complete ultrafilters
encompasses a major part of the theory of large cardinals,
and so the Ultrapower Axiom is a natural setting for large cardinal theory.\footnote{
    It is unknown, however, whether the Ultrapower Axiom is consistent
with large cardinals beyond a superstrong cardinal, or more generally
with large cardinals for which a canonical inner model theory has not been developed.
The author conjectures that the Ultrapower Axiom is consistent
with all large cardinal axioms.}

\cref{conj:main_conj} was originally motivated by the following consequence of UA:
\begin{thm*}[UA]\label{thm:ua_to_main_conj}
    For any countably compete ultrafilters \(U\) and \(W\), 
    the following are equivalent:
    \begin{enumerate}[(1)]
        \item \(U\times W\) is an ultrafilter.\label{item:ua:product}
        \item \(U\ltimes W = U\rtimes W\).\label{item:ua:commute}
        \item \(U\I W\) and \(W\I U\).\label{item:ua:cointernal}        
        \qed
    \end{enumerate}
\end{thm*}

We show here that this theorem can be proved for arbitrary ultrafilters
using GCH instead of UA:
\begin{repthm}{thm:GCH_to_main_conj}[GCH]
    For any ultrafilters \(U\) and \(W\), \(U\times W\) is an ultrafilter
    if and only if \(U\I W\) and \(W\I U\).
\end{repthm}
The proof of \cref{thm:GCH_to_main_conj}
is motivated by ideas from the UA theory,
and more specifically it makes use of
the two main tools from \cite{UA}, the Ketonen order and
the factorization of ultrafilters.
Arguably, the theorem could not have
proved without the help of the Ultrapower Axiom, 
even though the Ultrapower Axiom makes no appearance in the proof.
\section{Preliminaries}
\subsection{Ultrafilters and the Rudin-Keisler order}
If \(U\) is an ultrafilter on a set \(X\) and
\(f : X\to Y\) is a function, the \textit{pushforward}
of \(U\) by \(f\) is the ultrafilter \(f_*(U) = \{A\subseteq Y : f^{-1}[A]\in U\}\).
If \(W = f_*(U)\), then there is a unique elementary embedding \(k : M_W\to M_U\)
such that \(k\circ j_W = j_U\) and \(k(\id_W) = [f]_U\). Conversely,
if there is an elementary embedding \(k : M_W\to M_U\) such that \(k\circ j_W = j_U\),
then \(W = f_*(U)\) for some \(f\). In this case
(that is, if \(W\) is a pushforward of \(U\)), we say that \(W\) precedes
\(U\) in the \textit{Rudin-Keisler order,} and write \(W\RK U\).

Suppose \(M\) is a model of ZFC and \(X\in M\). An \textit{\(M\)-ultrafilter on \(X\)}
is an ultrafilter over the the Boolean algebra \(P^M(X)\). Given an \(M\)-ultrafilter \(U\) on \(X\),
one can form the associated ultrapower, denoted \(M_U^M\), by taking the quotient \(M^X\cap M\)
(the class of functions on \(X\) as computed in \(M\))
under the equivalence relation of \(U\)-almost everywhere equality
(which may not be definable over \(M\)). The equivalence class of \(f\in M^X\cap M\)
is denoted by \([f]_U^M\).
The associated ultrapower embedding
\(j_U^M : M\to M_U^M\) is given by \(j(s) = [c_s]_U^M\) 
where \(c_s : X\to \{s\}\) is the constant function.

Suppose \(j : M\to N\) is an elementary embedding.
For any \(X\in M\) and \(a\in N\) such that \(N\vDash a\in j(X)\), 
the \textit{\(M\)-ultrafilter on \(X\) derived from \(j\) using \(a\)},
is the ultrafilter \(D_X(j,a) = \{A\in P^M(X) : N\vDash a\in j(A)\}\).
Letting \(D = D_X(j,a)\), there is again a canonical factor embedding \(k : M_{D}^M\to N\),
denoted \(k_{a,j}\),
uniquely determined by requiring \(k\circ j_D = j\) and \(k(\id_D) = a\).
Now if \(j : M\to N\) is an ultrapower embedding, there is some \(X\in M\)
and \(a\in j(X)\) such that every element of \(N\) is definable in \(N\)
from parameters in \(j[M]\cup \{a\}\). Let \(D= D_X(j,a)\).
Then \(k_{a,j}: M_D^M\to N\) is surjective since its range contains 
\(j[M]\cup \{a\}\), and so \(k\) is an isomorphism
Note that if \(f_*(U) = W\), then \(W\) is the ultrafilter derived from \(j_U\) using \(a= [f]_U\),
and \(k_{a,j_U}\) is the canonical factor embedding.

\subsection{The symmetry of regularity}
Part of the motivation for this work was to
explain the connection between a theorem from Blass's thesis \cite[Theorem 3.6]{Blass}
and Kunen's Commuting Ultrapowers Lemma \cite[Lemma 1.1.25]{Larson}. These
lemmas concern two binary relations on ultrafilters 
that despite appearances turn out to be symmetric.
\begin{defn}
    Suppose \(F\) and \(G\) are filters on sets \(X\) and \(Y\).
    \begin{itemize}
        \item \(G\) is \textit{\(F\)-regular} if there is a sequence
        \(\langle B_x : x\in X\rangle\subseteq G\) such that for any \(A\in F^+\),
        \(\bigcap_{x\in A}B_x = \emptyset\).
        \item \(G\) is \textit{\(F\)-closed} if for any sequence
        \(\langle B_x : x\in X\rangle\subseteq G\), there is some \(A\in F\) such that
        \(\bigcap_{x\in A}B_x\in G\).
    \end{itemize}
\end{defn}
In this section, we show that our main question (\cref{qst:main_questions}) 
reduces to the problem of whether, given an ultrafilter \(U\), 
every \(U\)-nonregular ultrafilter is \(U\)-closed. 

The concept of regularity is a natural generalization of the classical concept
of a regular ultrafilter:
\begin{defn}
    If \(\kappa\leq \lambda\) are cardinals, 
    a filter \(G\) is \textit{\((\kappa,\lambda)\)-regular} 
    if there is a sequence \(\langle B_\alpha : \alpha < \lambda\rangle\) such that
    for any \(\sigma\subseteq \lambda\) with \(|\sigma|\geq \kappa\),
    \(\bigcap_{\alpha\in \sigma}B_\alpha = \emptyset\).
\end{defn}
In other words, \(G\) is \((\kappa,\lambda)\)-regular if
it is 
    \(F_{\kappa}(\lambda)\)-regular where \(F_{\kappa}(\lambda)\) 
    is the dual of the ideal 
    \(P_\kappa(\lambda) = \{\sigma \subseteq \lambda : |\sigma| < \kappa\}\).

The concept of closure is extracted from the following theorem from Blass's thesis:
\begin{thm}[Blass, {\cite[Theorem 3.6]{Blass}}]\label{thm:blass_completeness}
    Suppose \(U\) and \(W\) are ultrafilters on sets \(X\) and \(Y\).
    Then the following are equivalent:
    \begin{enumerate}[(1)]
        \item \label{item:product_uf}\(U\times W\) is an ultrafilter. 
        \item \label{item:U_closed}\(W\) is \(U\)-closed. 
        \item \label{item:W_closed}\(U\) is \(W\)-closed.
    \end{enumerate}
\end{thm}
Although the symmetry expressed in \cref{thm:blass_completeness} \ref{item:U_closed} and \ref{item:W_closed} 
does not extend to filters, one can deduce Blass's Theorem from a more general
filter-theoretic fact:
\begin{lma}\label{lma:filter_closure}
    Suppose \(F\) and \(G\) are filters. Then \(G\) is \(F\)-closed if and only if
    \(F\ltimes G = F\times G\).
    \begin{proof}
        In general, \(F\times G\subseteq F\ltimes G\), so it suffices to prove that
        the reverse inclusion is equivalent to the \(F\)-closure of \(G\).
        This is a routine matter of applying the correspondence between 
        binary relations \(R\subseteq X\times Y\) and \(X\)-indexed sequences
        \(\langle B_x : x\in X\rangle\subseteq Y\).
    \end{proof}
\end{lma}

\begin{proof}[Proof of \cref{thm:blass_completeness}]
    It suffices to prove the equivalence of \ref{item:product_uf} and \ref{item:U_closed},
    since the obvious symmetry of \ref{item:product_uf} in \(U\) and \(W\)
    then immediately implies the equivalence
    of \ref{item:U_closed} and \ref{item:W_closed}. 
    If \(U\times W\) is an ultrafilter, then since \(U\times W\) is 
    a maximal filter contained
    in \(U\ltimes W\), \(U\times W = U\ltimes W\), and hence \(W\) is \(U\)-closed
    by \cref{lma:filter_closure}. Conversely if \(W\) is \(U\)-closed,
    then \(U\times W = U\ltimes W\) by \cref{lma:filter_closure},
    and so since \(U\ltimes W\) is an ultrafilter, \(U\times W\) is.
\end{proof}
It turns out that regularity, like closure, is also a symmetric relation, but this time
the symmetry does extend to all filters. Despite the extensive
literature on regular ultrafilters, as far as the author knows, this symmetry
has never been observed. 
\begin{lma}\label{lma:regular_incompatible}
    Suppose \(F\) and \(G\) are filters. Then \(F\) is \(G\)-regular 
    if and only if \(F\ltimes G\) is incompatible with \(F\rtimes G\).
    \begin{proof}
        First assume \(F\) is \(G\)-nonregular.
        Fix \(R,S\subseteq X\times Y\) with \(R\in F\ltimes G\) and
        \(S\in F\rtimes G\), and let us show that \(R\cap S\) is nonempty.
        Take \(A\in F\) such that for all \(x\in A\), \((R)_x\in G\). 
        Note that since \(S\in F\rtimes G\), for \(G\)-almost all \(y\in Y\),
        \((S)^y\in F\), and hence \((S)^y\cap A\in F\). 
        Consider the sequence \(\langle (S)^y\cap A : y\in Y\rangle\). 
        Since \(F\) is \(G\)-nonregular, there is a \(G\)-positive
        set \(B\) such that \(\bigcap_{y\in B} (S)^y\cap A\neq \emptyset\).
        Fix \(x\in \bigcap_{y\in B} (S)^y\cap A\).
        Since \(x\in A\), \((R)_x\in G\), and therefore 
        \((R)_x\cap B\neq \emptyset\). Since \(x\in \bigcap_{y\in B} (S)^y\), 
        for any \(y\in (R)_x\cap B\),
        \(x\in (S)^y\); but then \((x,y)\in R\cap S\).

        Conversely, suppose \(F\) is \(G\)-regular, and
        let \(\langle A_y : y\in Y\rangle\subseteq F\) witnesses this.
        Let \(B_x = \{y\in Y : x\notin A_y\}\).
        We claim that \(B_x\in G\) for all \(x\in X\).
        This is because \(x\in \bigcap_{y\in Y\setminus B_x} A_y\),
        so \(Y\setminus B_x\notin G^+\), and hence
        \(B_x\in G\).
        But let \(\mathcal A = \{(x,y) : x\in A_y\}\) and \(\mathcal B = \{(x,y) : y\in B_x\}\).
        Then \(\mathcal A\in F\rtimes G\), \(\mathcal B \in F\ltimes G\), and 
        \(\mathcal A = (X\times Y)\setminus \mathcal B\). This shows \(F\rtimes G\) is incompatible with \(F\ltimes G\).
    \end{proof}
\end{lma}
\begin{thm}
    Suppose \(F\) and \(G\) are filters. Then \(F\) is \(G\)-regular
    if and only if \(G\) is \(F\)-regular.\qed
\end{thm}
For ultrafilters, we of course get more. 
\begin{thm}\label{thm:nonregular_commutes}
    If \(U\) and \(W\) are ultrafilters, then 
    \(U\) is \(W\)-nonregular if and only if \(U\ltimes W = U\rtimes W\).\qed
\end{thm}

Let us just mention one more familiar characterization of regularity.
\begin{defn}
    If \(F\) is a family of sets, the \textit{dual of \(F\)}
    is the family \(F^*\) of sets of the form \((\bigcup F)\setminus A\)
    for \(A\in F\). The \textit{fine
    filter on the dual of \(F\)}, denoted \(\mathbb F(F)\),
    is the filter on \(F^*\) generated by sets of the form
    \(\{\sigma\in F^* : x\in \sigma\}\) where \(x\in \bigcup F\).
\end{defn}
    The prototypical special case of this concept arises when 
    \(F = F_{\kappa}(\lambda)\) is the filter of subsets of \(\lambda\)
    with complement of cardinality less than \(\kappa\),
    so that \(\mathbb F(F)\) is the fine filter on \(P_\kappa(\lambda)\).
    Note that if \(\bigcap F\neq \emptyset\), then \(\mathbb F(F)\) is improper.

    The Katetov order is the natural generalization of the Rudin-Keisler order
    to filters:
\begin{defn}
    Suppose \(F\) is a filter on \(X\) and \(G\) is a filter on \(Y\). 
    The \textit{Katetov order} is defined by setting \(G\kat F\)
    if there is a function \(f : X\to Y\) such that \(G\subseteq f_*(F)\).
\end{defn}

\begin{prp}\label{prp:regular_fine}
    A filter \(F\) is \(G\)-regular if and only if \(\mathbb F(F)\kat G\).
    \begin{proof}
        Let \(X\) be the underlying set of \(F\).
        Notice that \(\langle A_y : y\in Y\rangle\) witnesses that \(F\) is
        \(G\)-regular if and only if the function \(f(y) = X\setminus A_y\)
        pushes \(G\) forward to \(\mathbb F(F)\).
    \end{proof}
\end{prp}

An immediate consequence of \cref{prp:regular_fine} is that
the \(F\)-regular filters are closed upwards in the Katetov order, or dually:
\begin{lma}\label{lma:regular_katetov}
    Suppose \(F, G\) and \(H\) are filters such that \(F\) is \(G\)-regular
    and \(G\kat H\).
    Then \(F\) is \(H\)-regular.\qed
\end{lma}
Put another way, if \(F\) is \(H\)-nonregular, then the class of \(F\)-regular filters
contains no Katetov predecessors of \(H\).

Finally, we explain the connection between the combinatorial structure
we have been exploring and Kunen's Commuting Ultrapowers Lemma.
\begin{defn}
If \(U\) is an ultrafilter, then \(\size{U}\) denotes the minimum cardinality
of a set in \(U\).
\end{defn}
Kunen's Commuting Ultrapowers Lemma is the following fact, which we have already considerably generalized.
\begin{thm}[Kunen, {\cite[Lemma 1.1.25]{Larson}}]\label{thm:kunen_commute}
    Suppose \(U\) and \(W\) are countably complete ultrafilters
    such that \(W\) is \(\size U^+\)-complete. Then 
    \(j_W(j_U) = j_U\restriction M_W\) and \(j_U(j_W) = j_W\restriction M_U\).
    \begin{proof}
        By \cref{lma:s_lemma}, \(j_U\restriction M_W\) is the ultrapower of \(M_W\) by
        \(s_W(U)\), while \(j_W(j_U)\) is clearly the ultrapower of \(M_W\) by \(j_W(U)\).
        So to show \(j_W(j_U) = j_U\restriction M_W\), it suffices to show
        \(s_W(U) = j_W(U)\), and similarly for \(j_U(j_W) = j_W\restriction M_U\).
        But since \(W\) is \(\size U^+\)-complete,
        \(W\) is \(U\)-closed. Hence \(U\times W\) is an ultrafilter,
        which implies \(U\ltimes W = U\rtimes W\), and so by \cref{thm:nonregular_commutes},
        \(s_W(U) = j_W(U)\) and \(s_U(W) = j_U(W)\).
    \end{proof}
\end{thm}
\subsection{Indecomposable ultrafilters}    
In this subsection, we prove the following theorems:
\begin{thm}\label{thm:indecomposable_product}
    Suppose \(U\) and \(W\) are ultrafilters such that for all cardinals \(\lambda\),
    either \(U\) or \(W\) is \((\lambda,2^\lambda)\)-indecomposable. 
    Then \(U\times W\) is an ultrafilter.
\end{thm}

\begin{thm}\label{thm:nonregular_indecomposable}
    Suppose \(U\) and \(W\) are ultrafilters such that \(U\ltimes W = U\rtimes W\).
    Then for all cardinals \(\lambda\),
    either \(U\) or \(W\) is \((\lambda,\lambda^+)\)-indecomposable.
\end{thm}

As an immediate corollary of \cref{thm:indecomposable_product} and
\cref{thm:nonregular_indecomposable},
we obtain our main theorem:
\begin{thm}[GCH]
    For any ultrafilters \(U\) and \(W\), \(U\ltimes W = U\rtimes W\)
    if and only if \(U\times W\) is an ultrafilter.\qed
\end{thm}

Recall the notion of indecomposability introduced above (\cref{defn:indecomposability}).
\begin{defn}
    If \(\lambda\) is a cardinal, \(F\) is \textit{\(\lambda\)-indecomposable} 
    if \(F\) is \((\lambda,\lambda)\)-indecomposable.
\end{defn}
It is easy to see that \(U\) is \((\lambda,\eta)\)-indecomposable (see \cref{defn:indecomposability})
if and only if
\(U\) is \(\gamma\)-indecomposable whenever \(\lambda\leq \gamma\leq \eta\).
We will use the term \textit{\(\lambda\)-decomposable} in the obvious way. (We avoid the
notion of \((\lambda,\eta)\)-decomposability in general, however, 
to avoid potential ambiguity.) We will use the following straightforward lemma:
\begin{lma}\label{lma:decomp_kat}
    A filter \(F\) is \(\lambda\)-decomposable
    if and only if \(F_\lambda(\lambda)\kat F\) where 
    \(F_\lambda(\lambda)\) is the Fr\'echet filter on \(\lambda\).\qed
\end{lma}
We start with the proof of \cref{thm:indecomposable_product}.
\begin{proof}[Proof of \cref{thm:indecomposable_product}]
Suppose \(U\times W\) is not an ultrafilter. We will show
that there is some cardinal \(\lambda\) such that
neither \(U\) nor \(W\) is \((\lambda,2^\lambda)\)-decomposable.
Let \(X\) be a set of minimal cardinality
carrying an ultrafilter \(\tilde{U}\RK U\) such that
\(\tilde{U}\times W\) is not an ultrafilter.
Let \(Y\) be a set of minimal cardinality
carrying an ultrafilter \(\tilde{W}\RK W\) such that
\(\tilde{U}\times \tilde{W}\) is not an ultrafilter.

Let \(\langle A_y : y \in Y\rangle\subseteq \tilde{U}\) witness
that \(\tilde{U}\) is not \(\tilde{W}\)-closed. 
Let \(f : Y \to P(X)\) be the function
\(f(y) = A_y\). Then \(\tilde{U}\) is not \(f_*(\tilde{W})\)-closed:
indeed, setting \(A_{B} = B\) if \(B\in \ran(f)\) and
\(A_B = X\) otherwise, the sequence
\(\langle A_B : B\in P(X)\rangle\) witnesses that \(\tilde{U}\)
is not \(f_*(\tilde{W})\)-closed. The minimality of the cardinality of
\(Y\) therefore implies \(|Y| \leq 2^{|X|}\).

Similarly, \(|X|\leq 2^{|Y|}\). Let \(\lambda = \min\{|X|,|Y|\}\).
Then \(\tilde U\) and \(\tilde W\) are both uniform ultrafilters
on sets whose cardinalities lie in the interval \([\lambda,2^\lambda]\).
Since \(\tilde U\RK U\) and \(\tilde W\RK W\),
this implies (by \cref{lma:decomp_kat})
that neither \(U\) nor \(W\) is \((\lambda,2^\lambda)\)-indecomposable. 
\end{proof}

We now turn to \cref{thm:nonregular_indecomposable},
which follows from a simple combinatorial
fact about the relationship between regularity and
decomposability.
\begin{cor}\label{cor:filter_indecomp_nonreg}
    Suppose \(F\) and \(G\) are filters such that \(F\) is \(G\)-nonregular
    and \(\lambda\) is a cardinal such that \(G\) is \(\lambda\)-decomposable.
    Then \(F\) is \((\lambda,\lambda)\)-nonregular.
    \begin{proof}
        Assume \(F\) is \((\lambda,\lambda)\)-regular, towards a contradiction.
        In other words, \(F\) is \(F_\lambda(\lambda)\)-regular
        where \(F_{\lambda}(\lambda)\) denotes the Fr\'echet filter on \(\lambda\).
        But \(G\) is \(\lambda\)-decomposable if and only if 
        \(F_\lambda(\lambda)\kat G\). So by \cref{lma:regular_katetov}, 
        \(F\) is \(G\)-regular, contrary to hypothesis.
    \end{proof}
\end{cor}

\cref{thm:nonregular_indecomposable} will be a consequence of \cref{cor:filter_indecomp_nonreg}
using the following theorem, considerably generalized by Lipparini \cite[Theorem 2.2]{LippariniRegular}:
\begin{lma}\label{lma:prikry}
    If \(U\) is \(\lambda^+\)-decomposable, either \(U\) is \(\cf(\lambda)\)-decomposable
    or \(U\) is \((\kappa,\lambda^+)\)-regular for some \(\kappa < \lambda\).
    In particular, if \(\lambda\) is regular, then \(U\) is \(\lambda\)-decomposable.\qed
\end{lma}

\begin{lma}\label{lma:nonregular_indecomposable}
    Suppose \(U\) is \((\lambda,\lambda)\)-nonregular. Then \(U\)
    is \((\lambda,\lambda^+)\)-indecomposable.
    \begin{proof}
        Clearly \(U\) is \(\cf(\lambda)\)-indecomposable (and
        in particular, \(U\) is \(\lambda\)-indecomposable).
        Therefore by \cref{lma:prikry}, if \(U\) were \(\lambda^+\)-decomposable,
        then \(U\) would \((\gamma,\lambda^+)\)-regular for some \(\gamma < \lambda\),
        but this of course would imply that \(U\) is \((\lambda,\lambda)\)-regular,
        contrary to assumption. Therefore \(U\) is \(\lambda^+\)-indecomposable.
        It follows that \(U\) is 
        \(\lambda\)-indecomposable and \(\lambda^+\)-indecomposable,
        so \(U\) is \((\lambda,\lambda^+)\)-indecomposable.
    \end{proof}
\end{lma}
\cref{thm:nonregular_indecomposable} is now immediate.
\begin{proof}[Proof of \cref{thm:nonregular_indecomposable}]
    Assume that \(W\) is not 
    \((\lambda,\lambda^+)\)-indecomposable, and we will show that \(U\) is.
    Since \(U\ltimes W= U\rtimes W\),
    \(U\) is \(W\)-nonregular. Therefore by \cref{cor:filter_indecomp_nonreg}, 
    if \(W\) is \(\lambda\)-decomposable,
    then \(U\) is \((\lambda,\lambda)\)-nonregular.
    It follows that \(U\) is \((\lambda,\lambda^+)\)-indecomposable
    by \cref{lma:nonregular_indecomposable}.
\end{proof}
\section{The internal relation}
\subsection{Amalgamations}
Our main questions concern the relationship between \(U\times W\)
and its extensions \(U\ltimes W\) and \(U\rtimes W\). More generally,
one can consider arbitrary ultrafilters extending of \(U\times W\),
which we will call \textit{amalgamations} of \((U,W)\).
An amalgamation \(D\) of \((U,W)\) is an upper bound of \(U\) and \(W\)
in the Rudin-Keisler order, since \(U = (\pi_0)_*(D)\) and \(W= (\pi_1)_*(D)\)
where \(\pi_n\) denotes projection to the \(n\)-th coordinate.

There is a second way to view amalgamations: as iterated ultrapowers.
Suppose \(U\) is an ultrafilter on \(X\) and \(W\) is an ultrafilter on \(Y\).
If \(D\) is an amalgamation of \((U,W)\), then
there is a unique elementary embedding \(k : M_U\to M_D\) such that
\(k\circ j_U = j_D\) and \(k(\id_U) = [\pi_0]_D\). 
This embedding is isomorphic to the ultrapower of \(M_U\) by an \(M_U\)-ultrafilter, 
defined in this context as follows: 
\begin{defn}
    Suppose \(U\) and \(W\) are ultrafilters
    on sets \(X\) and \(Y\) and \(D\) is an amalgamation of \((U,W)\).
    Then \((D)_0\) denotes the \(M_U\)-ultrafilter on 
    \(j_U(Y)\) derived from \(k\) using \([\pi_1]_D\)
    where \(k : M_U\to M_D\) is the canonical factor embedding.
\end{defn}
The ultrafilter \((D)_1\) is defined similarly, so
\((D)_1 = (\breve{D})_0\) where \(\breve{D}\) denotes the amalgamation
of \((W,U)\) given by \(D\).
One can compute:
\begin{lma}\label{lma:0_formula}
    If \(D\) is an amalgamation of \((U,W)\), then
    \[(D)_0 = \{[f_R]_U : R\in D\} \]
    where \(f_R :X\to P(Y)\) is given by 
    \(f_R(x) = (R)_x = \{y\in Y : (x,y)\in R\}\).\qed
\end{lma}
To prove that the canonical factor embedding \(k: M_U\to M_D\) is isomorphic to the
ultrapower embedding of \(M_U\) associated to \((D)_0\), 
one must show that the canonical factor embedding \(h : M_{(D)_0}^{M_U}\to M_D\)
is an isomorphism. It is enough to show \(h\) is surjective,
and this holds because every element of \(M_D\) is definable from parameters among \(j_D[V]\),
\([\pi_0]_D\), and \([\pi_1]_D\), and all these parameters lie in the range of \(h\):
\(j_D[V] = h\circ j_{(D)_0}^{M_U}\circ j_U[V]\),
\([\pi_0]_D = h(j_{(D)_0}^{M_U}(\id_U))\), and
\([\pi_1]_D = h(\id_{(D)_0}^{M_U})\).
\subsection{Amalgamations and the internal relation}
The internal relation is closely related to the ultrafilters
\(U\ltimes W\) and \(U\rtimes W\), viewed as
amalgamations of \((U,W)\).
This subsection is devoted to this relationship.
\begin{lma}\label{lma:UW_amalgamation}
    If \(U\) and \(W\) are ultrafilters, then \((U\ltimes W)_0 = j_U(W)\).\qed
\end{lma}
The ultrafilter \((U\ltimes W)_1\) is more interesting.
\begin{defn}
    If \(U\) and \(W\) are ultrafilters, then \(s_W(U) = (U\ltimes W)_1\).
\end{defn}
The following lemma shows that \(s_W(U)\) is in a sense
the canonical \(M_W\)-ultrafilter giving rise to the ultrapower
embedding \(j_U\restriction M_W\).
\begin{lma}\label{lma:s_lemma}
    Suppose \(U\) and \(W\) are ultrafilters on sets \(X\) and \(Y\).
    Then \(s_W(U)\) is the \(M_W\)-ultrafilter on \(j_W(X)\) derived from 
    \(j_U\restriction M_W : M_W\to j_U(M_W)\) using \([j_W]_U\),
    and the canonical factor embedding from \(M_{s_W(U)}^{M_W}\) to \(j_U(M_W)\) is an isomorphism.
    Finally, \(s_W(U)\) is given by the formula 
    \[s_W(U) = \{A\in j_W(P(X)) : j_W^{-1}[A]\in U\}\]
    \begin{proof}
        By definition, \(s_W(U)\) is the \(M_W\)-ultrafilter on \(j_W(X)\)
        derived from the canonical factor embedding \(i : M_W\to M_{U\ltimes W}\)
        using \([\pi_0]_{U\ltimes W}\). In general, the factor embedding
        \(p : M_{s_W(U)}^{M_W}\to M_{U\ltimes W}\) is an isomorphism;
        see the remarks following \cref{lma:0_formula}.

        We first show \(s_W(U)\) is the \(M_W\)-ultrafilter on \(j_W(X)\) derived from 
        \(j_U\restriction M_W\) using \([j_W]_U\).
        Let \(k : M_U\to M_{U\ltimes W}\) denote the canonical factor embedding.
        By \cref{lma:UW_amalgamation}, there is an isomorphism
        \(h : M_{U\ltimes W}\to M_{j_U(W)}^{M_U}\) such that
        \(h\circ k\) is equal to the ultrapower
        embedding of \(M_U\) associated to \(j_U(W)\) and 
        \(h([\pi_1]_{U\ltimes W}) = \id_{j_U(W)}\).

        We claim that \[h\circ i = j_U\restriction M_W\]
        Note that by elementarity \(j_U(M_W) = M_{j_U(W)}^{M_U}\),
        so \(h\circ i\) and \(j_U\restriction M_W\) at least have the same codomain.
        It suffices to show
        that \(h\circ i\) and \(j_U\) agree on the parameters
        \(j_W[V]\cup \{\id_W\}\), from which all 
        other elements of \(M_W\) are definable in \(M_W\).
        First, \[h\circ i \circ j_W = h\circ j_{U\ltimes W} = j_{j_U(W)}\circ j_U = j_U\circ j_W\]
        and this implies \(h\circ i\restriction j_W[V] = j_U\restriction j_W[V]\).
        Second, 
        \[h(i(\id_W)) = h([\pi_1]_{U\ltimes W}) = \id_{j_U(W)} = j_U(\id_W)\]
        This proves \(h\circ i = j_U\restriction M_W\).

        Now since \(s_W(U)\) is the \(M_W\)-ultrafilter on \(j_W(X)\)
        derived from the canonical factor embedding \(i : M_W\to M_{U\ltimes W}\)
        using \([\pi_0]_{U\ltimes W}\),
        \(s_W(U)\) is the \(M_W\)-ultrafilter on \(j_W(X)\) derived from \(h\circ i\)
        using \(h([\pi_0]_{U\ltimes W})\). But \(h\circ i = j_U\restriction M_W\) and
        \[h([\pi_0]_{U\ltimes W}) = h(k(\id_U)) = j_U(j_W)(\id_U) = [j_W]_U\]
        This proves \(s_W(U)\) is the \(M_W\)-ultrafilter on \(j_W(X)\) derived from 
        \(j_U\restriction M_W\) using \([j_W]_U\).
        
        The canonical factor embedding \(p : M_{s_W(U)}^{M_W}\to M_{U\ltimes W}\) is an isomorphism,
        and the embedding \(h : M_{U\ltimes W}\to j_U(M_W)\) is an isomorphism,
        so \(h\circ p: M_{s_W(U)}^{M_W}\to j_U(M_W)\) is an isomorphism, 
        and by uniqueness, \(h\circ p\)
        is the canonical factor embedding from \(M_{s_W(U)}^{M_W}\) to \(j_U(M_W)\).
        This proves that the canonical factor embedding from 
        \(M_{s_W(U)}^{M_W}\) to \(j_U(M_W)\) is an isomorphism.

        The pushforward \((j_W)_*(U) = \{A\subseteq j_W(X) : j_W^{-1}[A]\in U\}\)
        is as usual equal to the ultrafilter derived from \(j_U\) using \([j_W]_U\).
        Therefore \(s_W(U)\), being the ultrafilter derived from \(j_U\restriction M_W\) using
        \([j_W]_U\), is equal to \((j_W)_*(U)\cap M_W\).
        In other words, \[s_W(U) = \{A\in j_W(P(X)) : j_W^{-1}[A]\in U\}\qedhere\]
    \end{proof}
\end{lma}

From \cref{lma:s_lemma}, we immediately obtain:
\begin{lma}\label{defn:internal_relation2}
    For any ultrafilters \(U\) and \(W\), \(U\I W\) if and only if \(s_W(U)\in M_W\).\qed
\end{lma}
Given this equivalence, the internal relation 
looks quite a bit like the Mitchell order, 
the difference being that instead of
demanding that the \(V\)-ultrafilter \(U\) belong to \(M_W\),
one demands that the corresponding \(M_W\)-ultrafilter \(s_W(U)\) does.

The basic relationship between the internal relation and products
of ultrafilters is the following:
\begin{lma}\label{lma:nonregular_s}
    If \(U\) and \(W\) are ultrafilters, then \(U\) is \(W\)-nonregular
    if and only if \(s_W(U) = j_W(U)\).
    \begin{proof}
        Note that an amalgamation \(D\) of \((U,W)\) is uniquely determined
        by \(U\) and \((D)_0\) (or \(W\) and \((D)_1\)): by \cref{lma:0_formula},
        \[D = \{R\subseteq X\times Y : [f_R]_U\in (Z)_0\}\]
        Now \((U\rtimes W)_1 = (W\ltimes U)_0 = j_W(U)\) by \cref{lma:UW_amalgamation}
        and \((U\ltimes W)_1 = s_W(U)\) by definition, so
        \(s_W(U) = j_W(U)\) if and only if \(U\ltimes W = U\rtimes W\).
    \end{proof}
\end{lma}
\subsection{The internal relation and the Ketonen order}
The following theorem is almost proved in \cite{IR}, but here we are interested in
ultrafilters that may not be countably complete, so a new
argument is needed.
\begin{thm}\label{thm:cointernal_complete}
    Suppose \(U\) and \(W\) are ultrafilters such that \(U\I W\) and \(W\I U\).
    Then for any cardinal \(\delta\), 
    either \(U\) or \(W\) is \(\delta\)-indecomposable.
\end{thm}
\cref{thm:cointernal_indecomposable} below improves this theorem in the context of GCH,
but we do not know how to prove \cref{thm:cointernal_indecomposable} in ZFC.

To prove \cref{thm:cointernal_complete}, 
we introduce a version of the Ketonen order \cite{IR,KO} on countably incomplete ultrafilters:
\begin{defn}
    If \(X\) is a set, 
    a \textit{normed ultrafilter on \(X\)} is a pair \((U,f)\) where \(U\)
    is an ultrafilter on a set \(X\) and \(f: X\to \Ord\) is a function.
    If \(\delta\) is an ordinal, 
    \((U,f)\) is \textit{\(\delta\)-normed} if \(f(x) < \delta\) for \(U\)-almost all \(x\in X\).
\end{defn}
\begin{defn}
The \textit{Ketonen order} is defined on 
normed ultrafilters \((U,f)\) and \((W,g)\)
by setting \((U,f)\ke (W,g)\)
if there is some \(U_*\in M_W\) extending \(j_W[U]\) such that \((U_*,j_W(f))\)
is a \([g]_W\)-normed ultrafilter in \(M_W\).
\end{defn}
Equivalently, \((U,f)\ke (W,g)\) 
if there is a sequence \(\langle U_y : y\in Y\rangle\) of ultrafilters
such that \(U = W\ulim_{y\in Y} U_y\) and for \(W\)-almost all \(y\),
for \(U_y\)-almost all \(x\), \(f(x) < g(y)\).

If \(M\) and \(N\) are models of set theory and \(j : M\to N\) is an elementary embedding,
then \(j\) is \textit{close to \(M\)} if for all \(A\in N\), \(j^{-1}[A]\) is the 
extension of a set in \(M\).
\begin{lma}
    For normed ultrafilters \((U,f)\) and \((W,g)\), \((U,f)\ke (W,g)\)
    if and only if 
    there is a model of set theory \(N\) admitting an elementary embedding
    \(k : M_U\to N\) and a close embedding \(j : M_W\to N\) such that
    \(k\circ j_U = j\circ j_W\) and
    \(k([f]_U) < j([g]_W)\).
    \begin{proof}[Sketch]
        Assuming \((U,f)\ke (W,g)\), let \(j : M_W\to N\)
        be the ultrapower embedding of \(M_W\) associated to \(U_*\),
        and let \(k : M_U\to N\) map \([h]_U\) to \([j_W(h)]_{U_*}\).

        Conversely, given such embeddings \(j\) and \(k\), the ultrafilter \(U_*\)
        derived from \(j\) using \(k([f]_U)\) witnesses \((U,f)\ke (W,g)\).
    \end{proof}
\end{lma}

The Ketonen order is a strict partial order of the normed ultrafilters.
We leave the proof of transitivity to the reader, but we do show that it is
irreflexive.
\begin{thm}\label{thm:ke_irreflexive}
    \(\ke\) is irreflexive.
    \begin{proof}
        Suppose \((U,\varphi)\) is a normed ultrafilter on \(X\), and 
        suppose \(\langle U_x : x\in X\rangle\) satisfies
        \(U = U\ulim_{x\in X} U_x\). We must show that
        for \(U\)-almost all \(x\), for \(U_x\)-almost all \(y\),
        \(\varphi(x) \leq \varphi(y)\).

        Let \(B\) be the set of \(x\in X\) such that 
        for \(U_x\)-almost all \(y\in Y\),
        \(\{y\in X : \varphi(y) < \varphi(x)\}\in U_x\).
        Let \(\prec\) be the wellfounded relation on \(X\) 
        defined by setting \(y\prec x\) if \(\varphi(y) < \varphi(x)\).
        We define a set \(A\subseteq B\) by recursion on \(\prec\),
        putting \(x\in A\) if and only if \(\{y\in X:y\prec x\text{ and }y\in A\}\notin U_x\).
        By the definition of \(B\),
        for all \(x\in B\), \(\{y\in X : y\prec x\}\in U_x\),
        and therefore by the definition of \(A\), 
        \(x\in A\) if and only if \(A\notin U_x\).
        That is, \(A = \{x\in B : A\notin U_x\}\). As a consequence,
        \begin{align}
            A\in U&\iff \{x\in X: A\in U_x\}\in U\nonumber\\
                  &\iff \{x\in X: A\notin U_x\}\notin U\nonumber\\
                  &\,\hspace{.02cm}\implies \{x\in B: A\notin U_x\}\notin U\label{eq:B}\\
                  &\iff A\notin U\nonumber
        \end{align}
        The implication \cref{eq:B} cannot be reversed on pain of contradiction, so \(B\notin U\).
    \end{proof}
\end{thm}
    We will not use the following corollary, but it seems to have been unknown.
\begin{cor}\label{cor:dodd_jensen}
        Suppose \(U\) is an ultrafilter, \(j,k:M_U\to N\)
        are elementary embeddings with \(j\circ j_U = k\circ j_U\),
        and \(j\) is close to \(M_U\). Then for every ordinal \(a\) of \(M_U\),
        \(j(a) \leq k(a)\).\qed
\end{cor}
    
The analog of \cref{thm:ke_irreflexive}
proved in \cite{IR} applies to arbitrary
directed systems
of ultrafilters whose limit ultrapower is wellfounded.
\cref{thm:ke_irreflexive}, on the other hand, applies to a single 
ultrafilter but requires no assumption of wellfoundedness, or
equivalently of countable completeness.
The two theorems cannot be mutually generalized:
it simply is not true that the Ketonen order
on arbitrary directed systems of ultrafilters, defined following \cite{IR},
is irreflexive.
For example, there is a directed system of ultrafilters \(E\)
whose limit ultrapower \(M_E\)
admits a nontrivial elementary self embedding 
\(k : M_E\to M_E\)
such that \(k\circ j_E = j_E\),
yet \(k\) is isomorphic to an internal ultrapower
embedding of \(M_E\). This means
\cref{cor:dodd_jensen} does not generalize to
\(E\), and hence \cref{thm:ke_irreflexive}
cannot either. 

We now explain the connection between the Ketonen order and the internal relation.
\begin{defn}
    If \((W,g)\) is a normed ultrafilter, then \(\delta(W,g)\) denotes the least
    ordinal \(\delta\) such that \((W,g)\) is \(\delta\)-normed.
\end{defn}
\begin{thm}\label{thm:I_ke}
    If \(U\I W\), \(\delta = \delta(W,g)\) is a limit ordinal, 
    and \((U,f)\) is \(\delta\)-normed, then \((U,f)\ke (W,g)\).
    \begin{proof}
        For any \(\epsilon < \delta\), for \(W\)-almost all \(y\), \(g(y) > \epsilon\).
        Therefore in \(M_U\), for any \(\epsilon < j_U(\delta)\), for \(j_U(W)\)-almost all \(y\),
        \(j_U(g)(y) > \epsilon\). It follows that for \(j_U(W)\)-almost all \(y\),
        \(j_U(g)(y) > [f]_U\). In other words,
        for \(U\)-almost all \(x\), for \(W\)-almost all \(y\), \(f(x) < g(y)\). 
        This means \(\{x : j_W(f(x)) < [g]_W\}\in U\),
        or equivalently, \(\{x : j_W(f)(x) < [g]_W\}\in s_W(U)\).
        This means that \((s_W(U),j_W(f))\) is a \([g]_W\)-normed ultrafilter in \(M_W\),
        and hence \(s_W(U)\) witnesses that \((U,f)\ke (W,g)\).
    \end{proof}
\end{thm}

We now prove our first indecomposability theorem for the internal relation.
\begin{proof}[Proof of \cref{thm:cointernal_complete}]
    Assume towards a contradiction that the theorem fails, 
    and neither \(U\) nor \(W\) is \(\delta\)-indecomposable.
    Let \(X\) and \(Y\) be the underlying sets of \(U\) and \(W\).
    The \(\delta\)-decomposability of \(U\) yields
    a function \(f : X\to \delta\) such that \(f\) is not bounded below \(\delta\)
    on a set in \(U\). In other words, \(\delta(U,f) = \delta\).
    Similarly, for some \(g : X \to \delta\), \(\delta(W,g) = \delta\).
    By \cref{thm:I_ke}, since \(U\I W\) and 
    \((U,f)\) is \(\delta\)-normed, \((U,f)\ke (W,g)\).
    The same logic shows \((W,g)\ke (U,f)\). By the transitivity of the Ketonen order,
    \((U,f)\ke (U,f)\), and this contradicts \cref{thm:ke_irreflexive}.
\end{proof}
The following corollary is proved more directly in \cite{UA}.
\begin{cor}\label{cor:internal_irreflexive}
    Any ultrafilter \(U\) such that \(U\I U\) is principal.
    \begin{proof}
        Suppose \(U\) is nonprincipal. Let \(\kappa\)
        be the completeness of \(U\), the supremum of all 
        cardinals \(\delta\) such that \(U\) is \(\delta\)-complete. 
        Then \(\kappa\) is regular 
        and \(U\) is \(\kappa\)-decomposable. But \(U\I U\),
        so \cref{thm:cointernal_complete} implies \(U\) is
        not \(\kappa\)-decomposable, and this is a contradiction.
    \end{proof}
\end{cor}
\subsection{\((\lambda,\lambda^+)\)-indecomposability for mutually internal ultrafilters}
This final subsection is devoted to the following theorem:
\begin{thm}[GCH]\label{thm:cointernal_indecomposable}
    Suppose \(U\) and \(W\) are countably complete ultrafilters
    such that \(U\I W\) and \(W\I U\). Then for all cardinals \(\lambda\),
    either \(U\) or \(W\) is \((\lambda,\lambda^+)\)-indecomposable.
\end{thm}

Combining this with \cref{thm:indecomposable_product} immediately yields
the last of our main results:
\begin{thm}[GCH]\label{thm:GCH_to_main_conj}
    Suppose \(U\) and \(W\) are ultrafilters
    such that \(U\I W\) and \(W\I U\). Then \(U\times W\) is an ultrafilter.\qed
\end{thm}

Under UA, one can easily prove that if \(U\I W\) and \(W\I U\),
then \(U\) is \(W\)-nonregular.
This fact is unique among UA results in that it applies to arbitrary ultrafilters,
so we sketch a proof for the reader familiar with \cite{UA}.
\begin{thm}[UA]
    Suppose \(U\) and \(W\) are ultrafilters such that \(U\I W\) and \(W\I U\).
    Then \(U\) is \(W\)-nonregular.
    \begin{proof}
        At least one of \(U\) and \(W\) is countably complete 
        by \cref{thm:cointernal_complete}. So without loss of generality,
        assume that \(U\) is.

        Applying UA in \(M_W\), the ultrapower embeddings
        associated to \(s_W(U)\) and \(j_W(U)\) admit an internal
        ultrapower comparison. In other words, there exist
        internal ultrapower embeddings \(k : j_U(M_W)\to N\) and \(\ell : j_U(M_W)\to N\)
        such that \(k\circ j_U\restriction M_W = \ell\circ j_W(j_U)\).
        To conclude that \(s_W(U) = j_W(U)\) (and hence that \(U\) is \(W\)-nonregular), 
        it suffices to show that \(k(\id_{s_W(U)}) = \ell(\id_{j_W(U)})\).

        We have that \(k(\id_{s_W(U)}) = k(j_U(j_W)(\id_U))\)
        while \(\ell(\id_{j_W(U)}) = \ell(j_W(\id_U))\).
        Now \(\ell\circ j_W\) and \(k\circ j_U(j_W)\) are both internal
        ultrapower embeddings from \(M_U\) to \(N\),
        and moreover \(\ell\circ j_W \circ j_U =  k\circ j_U(j_W)\circ j_U\).
        By the uniqueness of such embeddings (\cite[Theorem ??]{UA}),
        \(\ell\circ j_W=k\circ j_U(j_W)\) and in particular
        \(\ell\circ j_W(\id_U) =k\circ j_U(j_W)(\id_U)\).
        It follows that \(k(\id_{s_W(U)}) = \ell(\id_{j_W(U)})\),
        and this proves the theorem.
    \end{proof}
\end{thm}

We now turn to the proof of \cref{thm:cointernal_indecomposable} itself,
where we seem to need a completely different argument.
This will require some lemmas on the interaction between the Rudin-Keisler order and the
internal relation:
\begin{thm}\label{thm:internal_RK}
    Suppose \(\bar U\RK U\) and \(W\) are ultrafilters.
    \begin{enumerate}[(1)]
        \item \label{item:s_pushforward} If \(\bar U = f_*(U)\), then 
        \(s_W(\bar U) = j_W(f)_*(s_W(U))\). In particular, if \(U\I W\), then \(\bar U\I W\). 
        \item \label{item:s_pullback} If \(k : M_{\bar U}\to M_U\)
        is an elementary embedding such that \(k\circ j_{\bar U} = j_U\),
        then \(s_{\bar U}(W) = k^{-1}[s_U(W)]\). 
        In particular, if \(W\I U\) and \(s_U(W)\in \ran(k)\), then \(W\I \bar U\).
    \end{enumerate}
    \begin{proof}
        For \ref{item:s_pushforward}, recall that for any \(Z\), \(s_W(Z) = (j_W)_*(Z)\cap M_W\) and observe: 
        \[j_W(f)_*((j_W)_*(U)) = (j_W(f)\circ f)_*(U) = (j_W\circ f)_*(U)= (j_W)_*(f_*(U)) = (j_W)_*(\bar U)\]

        For \ref{item:s_pullback}, note that \(A\in k^{-1}[s_{U}(W)]\) if and only if
        \(j_{U}^{-1}[k(A)]\in W\), but \(j_{U}^{-1}[k(A)] = j_{\bar U}^{-1}[A]\),
        so \(j_{U}^{-1}[k(A)]\in W\) if and only if \(j_{\bar U}^{-1}[A]\in W\) or equivalently
        \(A\in s_{\bar U}(W)\).
    \end{proof}
\end{thm}
\begin{cor}\label{lma:push_commute} If \(U\ltimes W = U\rtimes W\), then for any \(\bar
    U\RK U\) and \(\bar W \RK W\), \(\bar U\ltimes \bar W = \bar U\rtimes \bar W\).
    \begin{proof}
        It suffices to show that \(\bar U\ltimes W = \bar U\rtimes W\).
        For this we must show that \(j_{\bar U}(W) = s_{\bar U}(W)\).
        Since \(\bar U \RK U\), there is an elementary embedding
        \(k : M_{\bar U}\to M_U\) such that \(k\circ j_{\bar U} = j_U\).
        Therefore \(s_U(W) = j_U(W) = k(j_{\bar U}(W))\). 
        Hence \(j_{\bar U}(W) = k^{-1}[s_U(W)] = s_{\bar U}(W)\), applying \cref{thm:internal_RK}.
        Now by \cref{lma:nonregular_s}, \(\bar U\ltimes W = \bar U\rtimes W\).
    \end{proof}
\end{cor}

We also need a lemma that will allow us to generate an ultrafilter
extending an \(M\)-ultrafilter:
\begin{lma}\label{lma:filter_completeness}
    Suppose \(\lambda\) is a strong limit cardinal,
    \(M\) is a model with the \(\lambda\)-cover property,
    and \(U\) is an \(M\)-\(\lambda\)-complete \(M\)-ultrafilter.
    Then \(U\) generates a \(\lambda\)-complete filter.
    \begin{proof}
        Suppose \(\sigma\subseteq U\)
        and \(|\sigma| < \lambda\). We must show
        \(\bigcap \sigma\neq \emptyset\).
        
        Using the \(\lambda\)-cover property,
        fix \(\tau\in M\) such that \(\sigma\subseteq \tau\)
        and \(|\tau| < \lambda\).
        Notice \(\tau\cap U\in M\). 
        First, the \(M\)-\(\lambda\)-completeness of \(U\) implies
        \(j_U[\tau\cap U] = 
        \{A\in j_U(\tau) : \id_U\in A\} \in M_U^M\).
        But since \(2^{|\tau|} < \lambda\),
        the \(M\)-\(\lambda\)-completeness of \(U\) implies
        \(j_U : P^M(\tau)\to j_U(P^M(\tau))\) is an isomorphism, and so
        \(\tau \cap U = j_U^{-1}[j_U[\tau\cap U]]\in M\).

        Since \(U\) is \(M\)-\(\lambda\)-complete,
        \(\bigcap (\tau\cap U)\neq \emptyset\).
        Since \(\sigma\subseteq (\tau\cap U)\),
        \(\bigcap \sigma\neq \emptyset\).
    \end{proof}
\end{lma}
\begin{lma}[Kunen]\label{lma:regular_RK}
    Suppose \(W\) is \((\kappa,\delta)\)-regular and
    \(F\) is a \(\kappa\)-complete filter generated by at most \(\delta\)-many sets. 
    Then \(F\kat W\).
    \begin{proof}
        Let \(\mathcal B\) be a basis for \(F\) with \(|\mathcal B| \leq \delta\).
        Since \(W\) is \((\kappa,\delta)\)-regular,
        the filter \(G\) on \(P_\kappa(\mathcal B)\) generated by sets of
        the form \(\{\sigma \in P_\kappa(\mathcal B) : A\in \sigma\}\) for
        \(A\in \mathcal B\) lies below \(W\) in the Katetov order.
        Therefore it suffices to show that \(F\kat G\).
        But let \(f : P_\kappa(\mathcal B) \to X\) be any function such that
        \(f(\sigma)\in \bigcap \sigma\). Easily, \(F\subseteq f_*(G)\), as desired.
    \end{proof}
\end{lma}
Finally, we need a theorem due essentially to Silver.
\begin{defn}
    Suppose \(U\) and \(W\) are ultrafilters on sets \(X\)
    and \(Y\).
    Then \(U\RK^\gamma W\) if there is a function
    \(p :Y\to X\) such that for all \(\beta < \gamma\),
    for all \(f : Y\to \beta\), \(f\) factors through \(p\)
    modulo \(W\): there is a \(g : X\to \beta\)
    such that \([f]_W = [g\circ p]_W\). 
\end{defn}

\begin{thm}[{Silver, \cite{Silver}}]\label{thm:Silver}
    Suppose \(\lambda\) is a regular cardinal
    and \(U\) is a \((\lambda,\eta)\)-indecomposable ultrafilter
    and \(2^\lambda \leq \eta\).
    Then there is an ultrafilter \(D\) on a cardinal less than \(\lambda\)
    such that \(D\RK^{\eta} U\).\qed
\end{thm}

\begin{proof}[Proof of \cref{thm:cointernal_indecomposable}]
    We assume by induction that the theorem holds for all \(\bar \lambda < \lambda\),
    and towards a contradiction that there exist ultrafilters \(U\) and \(W\) 
    such that \(U\I W\), \(W\I U\), 
    but neither \(U\) nor \(W\) is \((\lambda,\lambda^+)\)-indecomposable.
    By \cref{lma:prikry}, we may assume without loss of generality
    that \(\lambda\) is singular, \(W\) is \(\lambda\)-decomposable, and 
    \(U\) is \((\kappa,\lambda^+)\)-regular
    for some \(\kappa < \lambda\).

    We claim one can reduce to the case that \(W\) is a uniform ultrafilter on \(\lambda\).
    Note that \(W\) is \((\lambda^+,\lambda^{+\omega})\)-indecomposable
    since \(W\) is \(\lambda^+\)-indecomposable (by \cref{lma:prikry}).
    Applying \cref{thm:Silver} and the GCH, 
    let \(\tilde{W}\) be a uniform ultrafilter on \(\lambda\)
    such that \(\tilde{W}\RK^{\lambda^{+\omega}} W\).
    Let \(p\) witness that \(\tilde{W}\RK^{\lambda^{+\omega}} W\), so
    \(p_*(W) = \tilde{W}\).
    Let \(\tilde{U}\) be the ultrafilter on 
    \(\beta(\lambda)\times P_{\kappa}(\lambda)\) 
    derived from \(j_U\) using \((s_U(\tilde W),\sigma)\) 
    where \(\sigma\in j_U(P_\kappa(\lambda^+))\)
    covers \(j_U[\lambda^+]\). (Here \(\beta(\lambda)\)
    denotes the set of ultrafilters on \(\lambda\).) 
    Thus by \cref{thm:internal_RK},
    \(\tilde W\I \tilde U\) and by \cref{lma:regular_katetov},
    \(\tilde U\) is \((\kappa,\lambda^+)\)-regular.
    Let \(f\) be such that \([f]_W = s_W(\tilde{U})\), 
    and assume without loss of generality
    that the range of \(f\) has cardinality
    strictly less than \(\lambda^{+\omega}\).
    Since \(\tilde{W} \RK^{\lambda^{+\omega}} W\),
    there is a function \(\tilde{f}\) such that 
    \(\tilde{f}\circ p = f\). By \cref{thm:internal_RK}, 
    \([\tilde{f}]_{\tilde{W}} = s_{\tilde{W}}(\tilde{U})\).
    So \(\tilde U \I \tilde W\).
    Replacing \(U\) and \(W\) with \(\tilde U\) and \(\tilde W\),
    we reduce to the case that \(W\) is a uniform ultrafilter on \(\lambda\).

    Since \(U\) is \((\kappa,\lambda^+)\)-regular, 
    \(U\) is \(\delta\)-decomposable for every regular cardinal \(\delta\) such that 
    \(\kappa\leq\delta < \lambda\). (This is a well-known fact,
    but it is also an immediate 
    consequence of \cref{lma:decomp_kat} and \cref{lma:regular_RK}.)
    Therefore by our induction hypothesis, \(W\) is \(\eta\)-indecomposable for all
    cardinals \(\eta\) such that \(\kappa\leq \eta < \lambda\):
    we know at least one of \(U\) and \(W\) is \((\eta,\eta^+)\)-indecomposable,
    and \(U\) is not (since \(U\) is \(\eta^+\)-decomposable since \(\eta^+\) is regular).
    We are assuming GCH, so we can apply Silver's Lemma (\cref{thm:Silver})
    to obtain an ultrafilter \(D\) on a cardinal \(\gamma\) less than \(\kappa\) 
    such that \(D\leq_{\lambda} W\).
    Equivalently, there is an elementary embedding \(k : M_D\to M_W\) such that
    \(k\circ j_D = j_W\) and
    \(j_W(\eta)\subseteq \ran(k)\) for all \(\eta < \lambda\).
    
    By our induction hypothesis,
    for any cardinal \(\bar \lambda < \lambda\), either \(U\) or \(W\) is 
    \((\bar \lambda,(\bar \lambda)^+)\)-in\-de\-com\-pos\-able.
    Since \(D\RK W\) and \(\size D < \lambda\),
    for any cardinal \(\bar \lambda < \lambda\), either \(U\) or \(D\) is 
    \((\bar \lambda,(\bar \lambda)^+)\)-in\-de\-com\-pos\-able.
    Applying \cref{thm:indecomposable_product}, \(D\times U\) is an ultrafilter.
    In particular, \(j_D(U) = s_D(U)\). 
    Let \(k : M_D\to M_{W}\) be an elementary embedding
    such that \(k \circ j_D = j_{\bar W}\) and \(j_U(\eta) \subseteq \ran(k)\) for all
    \(\eta < \lambda\). Let \(W_*\) be the \(M_D\)-ultrafilter on \(j_D(\lambda)\) derived 
    from \(k\) using \(\id_{W}\). Then \(M_{W}\) is isomorphic to the ultrapower of \(M_D\)
    by \(W_*\) and \(k\) to the associated ultrapower embedding.
    We will identify the two models via the unique isomorphism between \(k\) and 
    \((j_{W_*})^{M_D}\), so \(M_W = (M_{W_*})^{M_D}\) and \(k = (j_{W_*})^{M_D}\).

    Let \(F\) be the filter generated by \(W_*\) (in \(V\)).
    By \cref{lma:filter_completeness},
    \(F\) is \(\lambda\)-complete, 
    and since \(|j_D(P(\lambda))|\leq (\lambda^+)^\gamma = \lambda^+\),
    \(F\) is generated by at most \(\lambda^+\)-many sets.
    Therefore by \cref{lma:regular_RK}, \(F\) extends to an ultrafilter \(G\)
    preceding the \((\kappa,\lambda^+)\)-regular ultrafilter 
    \(U\) in the Rudin-Keisler order. Since \(G\RK U \I D\),
    \cref{thm:internal_RK} yields that \(s_D(G)\RK s_D(U)\) in \(M_D\),
    and in particular \(G\I D\). Note that \(W_*\) is the ultrafilter
    derived from \(j_G\restriction M_D\) using \(\id_G\). Therefore 
    \(W_*\in M_D\) and \(W_*\RK s_D(G)\) in \(M_D\).

    So far we have shown that in \(M_D\), \[W_*\RK s_D(G)\RK s_D(U)\] Since \(U\I W\),
    \(s_W(U)\in M_W\). But \(s_W(U) = s_{W_*}(s_D(U))\), so 
    since \(s_{W_*}(s_D(U))\in M_W = (M_{W_*})^{M_D}\),
    \(s_D(U)\I W_*\) in \(M_D\). Now in \(M_D\), 
    \(W_*\RK s_D(U)\I W_*\). Applying \cref{thm:internal_RK} one last time,
    \(W_*\I W_*\) in \(M_D\), and it follows
    (by \cref{cor:internal_irreflexive})
    that \(W_*\) is a principal ultrafilter of \(M_D\).
    Thus \(j_D = j_{W_*}^{M_D}\circ j_D = j_W\), or in other words,
    \(D\RKE W\), contradicting that \(\size D < \size W\).
\end{proof}
\bibliographystyle{plain}
\bibliography{Bibliography.bib}

\begin{thebibliography}{1}

\bibitem{Blass}
Andreas~Raphael Blass.
\newblock {\em Orderings of {U}ltrafilters}.
\newblock ProQuest LLC, Ann Arbor, MI, 1970.
\newblock Thesis (Ph.D.)--Harvard University.

\bibitem{UA}
Gabriel Goldberg.
\newblock {\em The {U}ltrapower {A}xiom}.
\newblock PhD thesis, Harvard University, 2019.

\bibitem{IR}
Gabriel Goldberg.
\newblock Rank-into-rank embeddings and {S}teel's conjecture.
\newblock {\em Journal of Symbolic Logic}, 2021.

\bibitem{Larson}
Paul~B. Larson.
\newblock {\em The stationary tower}, volume~32 of {\em University Lecture
  Series}.
\newblock American Mathematical Society, Providence, RI, 2004.
\newblock Notes on a course by W. Hugh Woodin.

\bibitem{LippariniRegular}
Paolo Lipparini.
\newblock More on regular and decomposable ultrafilters in {ZFC}.
\newblock {\em Mathematical Logic Quarterly}, 56(4):340--374, Jul 2010.

\bibitem{Silver}
Jack~H. Silver.
\newblock Indecomposable ultrafilters and {$0^\#$}.
\newblock In {\em Proceedings of the {T}arski {S}ymposium ({P}roc. {S}ympos.
  {P}ure {M}ath., {V}ol. {XXV}, {U}niv. {C}alif., {B}erkeley, {C}alif., 1971)},
  pages 357--363, 1974.

\end{thebibliography}
\end{document}